\documentclass{amsart}
\usepackage[all]{xy}

\theoremstyle{plain}
\newtheorem{theorem}{Theorem}
\newtheorem{proposition}{Proposition}

\theoremstyle{definition}
\newtheorem{definition}{Definition}
\newtheorem{example}{Example}

\newcommand{\R}{{\mathbb R}}

\begin{document}

\title[Encouraging the grand coalition]{Encouraging the grand coalition in \\convex cooperative games}

\author{Titu Andreescu}

\address{The University of Texas at Dallas
Science/Mathematics Education Department, PO Box 830688 Mail Station
FN33, Richardson TX 75083-0688, USA}

\author{Zoran \v{S}uni\'c}

\thanks{The second author is partially supported by NSF grant DMS-0600975}

\address{Department of Mathematics, Texas A\&M University, College Station, TX 77843-3368, USA}



\begin{abstract}
A solution function for convex transferable utility games encourages
the grand coalition if no player prefers (in a precise sense defined
in the text) any coalition to the grand coalition. We show that the
Shapley value encourages the grand coalition in all convex games and
the $\tau$-value encourages the grand coalitions in convex games up
to three (but not more than three) players. Solution functions that
encourage the grand coalition in convex games always produce
allocations in the core, but the converse is not necessarily true.
\end{abstract}

\maketitle

%-------------------------------------------------------------

\section{Cooperative games}

We begin by recalling the main concepts and their basic properties.
The notation mostly follows~\cite{curiel:b-games}
and/or~\cite{branzei-d-t:b-games}.

Let $N=\{1,\dots,n\}$. The elements of $N$ are called
\emph{players}, its subsets are called \emph{coalitions}, and the
set $N$ is called the \emph{grand coalition}. A \emph{cooperative
transferable utility game} with $n$-players is a function $v: 2^N
\to \R$ such that $v(\emptyset)=0$, where $2^N$ is the set of all
subsets of $N$.

For a given game $v$, we often denote $v\left(\{i\}\right)$ by
$v(i)$ or $v_i$. More generally, for any function $x:N \to \R$ and
$i \in N$, we denote $x(i) = x_i$. Thus we (sometimes) think of
functions $x:N \to \R$ as vectors in $\R^n$.  For a function $x: N
\to \R$ and a coalition $A \subseteq N$, we write $x(A) = \sum_{j
\in A}x_j$.

A game $v:2^N \to \R$ is called \emph{super-additive} if, for all
disjoint coalitions $A,B \subseteq N$,
\[ v(A) + v(B) \leq  v(A \cup B) \]
and is called \emph{convex} if, for all coalitions $A,B \subseteq
N$,
\[ v(A) + v(B) \leq  v(A \cup B) + v(A \cap B). \]

\begin{example}\label{e:4-player}
Define a 4-player game $v$ on $N=\{1,2,3,4\}$ by the diagram in
Figure~\ref{f:4-player} (the value of each coalition is provided at
the vertex representing the coalition).
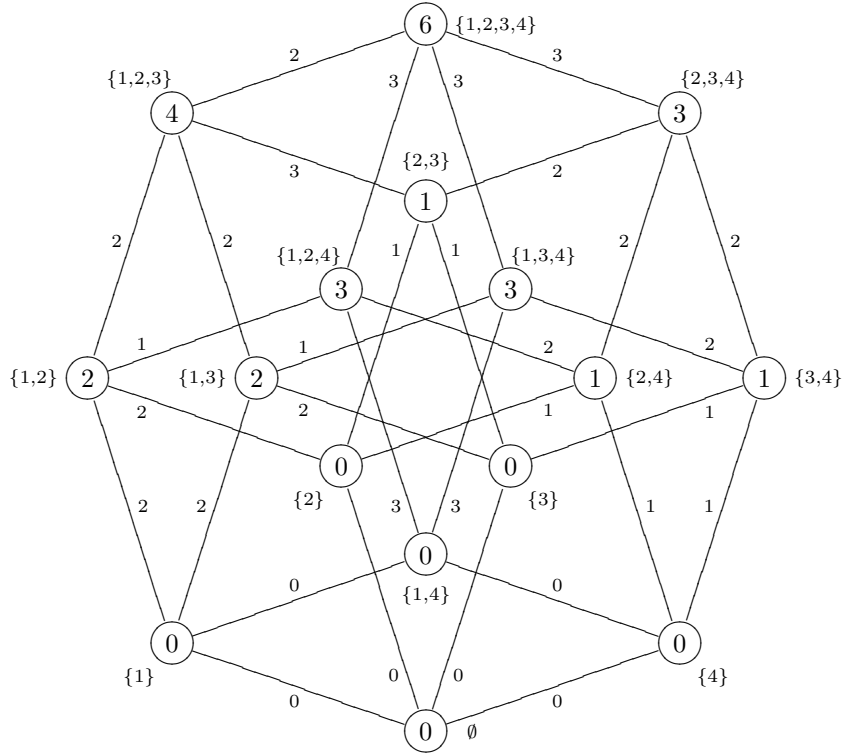
\begin{figure}[!ht]
\[
\xymatrix@C=15pt@R=15pt{
 &&&&& *++[o][F-]{6} \ar@{-}[dlll]_{2} \ar@{-}[drrr]^{3} \ar@{}[r]|<<<<<<<{\{1,2,3,4\}}
                                 \ar@{-}[dddl]_<<<<<<{3} \ar@{-}[dddr]^<<<<<<{3} &&&&
 \\
 &&*++[o][F-]{4} \ar@{-}[dddl]_{2} \ar@{-}[dddr]^{2} \ar@{-}[drrr]_{3} \ar@{}[ul]|<<<<{\{1,2,3\}}&&&&&&
  *++[o][F-]{3} \ar@{-}[dddr]^{2} \ar@{-}[dddl]_{2} \ar@{-}[dlll]^{2} \ar@{}[ur]|<<<<{\{2,3,4\}}&
 \\
 &&&&& *++[o][F-]{1} \ar@{-}[dddl]_<<<<<{1} \ar@{-}[dddr]^<<<<<{1} \ar@{}[u]|<<<{\{2,3\}}&&&&
 \\
 &&&& *++[o][F-]{3} \ar@{-}[dlll]_>>>>>>{1} \ar@{-}[drrr]^>>>>>{2} \ar@{-}[dddr]_>>>>>{3} \ar@{}[ul]|<<<<{\{1,2,4\}}&&
     *++[o][F-]{3} \ar@{-}[dlll]_>>>>>{1} \ar@{-}[drrr]^>>>>>>{2} \ar@{-}[dddl]^>>>>>{3} \ar@{}[ur]|<<<<{\{1,3,4\}}&&&
 \\
 &*++[o][F-]{2} \ar@{-}[dddr]^{2} \ar@{-}[drrr]_<<<<<<{2} \ar@{}[l]|<<<<<{\{1,2\}}&&
 *++[o][F-]{2} \ar@{-}[dddl]_{2} \ar@{-}[drrr]_<<<<<{2} \ar@{}[l]|<<<<<{\{1,3\}} &&&&
 *++[o][F-]{1} \ar@{-}[dlll]^<<<<<{1} \ar@{-}[dddr]^{1} \ar@{}[r]|<<<<<{\{2,4\}}&&
 *++[o][F-]{1} \ar@{-}[dlll]^<<<<<<{1} \ar@{-}[dddl]_{1}
 \ar@{}[r]|<<<<<{\{3,4\}} &
 \\
 &&&& *++[o][F-]{0} \ar@{-}[dddr]_>>>>>>{0} \ar@{}[dl]|<<<<{\{2\}}&&
     *++[o][F-]{0} \ar@{-}[dddl]^>>>>>>{0} \ar@{}[dr]|<<<<{\{3\}}&&&
 \\
 &&&&& *++[o][F-]{0} \ar@{-}[dlll]_{0} \ar@{-}[drrr]^{0} \ar@{}[d]|<<<{\{1,4\}}&&&&
 \\
 && *++[o][F-]{0} \ar@{-}[drrr]_{0} \ar@{}[dl]|<<<<{\{1\}} &&&&&&
   *++[o][F-]{0} \ar@{-}[dlll]^{0} \ar@{}[dr]|<<<<{\{4\}}&
 \\
 &&&&&*++[o][F-]{0} \ar@{}[r]|<<<<{\emptyset}&&&&
}
\]
\caption{A convex 4-player game}\label{f:4-player}
\end{figure}
The same game is given in a tabular form in Table~\ref{t:4-player}.
\begin{table}[!ht]
\[
\begin{array}{cc|cc|cc|cc}
 A & v(A) & A & v(A) & A & v(A) & A & v(A) \\
 \hline
 \{1\} & 0 & \{2\} & 0 & \{3\} & 0 & \{4\} & 0 \\
 \{1,2\} & 2 & \{1,3\} & 2 & \{1,4\} & 0 & \{2,3\} & 1 \\
 \{2,4\} & 1 & \{3,4\} & 1 & \{1,2,3\} & 4 & \{1,2,4\} & 3 \\
 \{1,3,4\} & 3 & \{2,3,4\} & 3 & N & 6 & \emptyset & 0
\end{array}
\]
\caption{A convex 4-player game}\label{t:4-player}
\end{table}

It is straightforward to check that the game $v$ is convex.
\end{example}

A way to interpret cooperative games is as follows. Assume that the
players in the set $N$ can form various coalitions each of which has
value prescribed by $v$ (say $v(A)$ represents the amount the
coalition $A$ can earn by cooperating). The super-additivity
condition implies that ``the whole is larger than the sum of its
parts'', i.e., forming larger coalitions positively affects the
value. The convexity condition is just a stronger form of the
super-additivity condition. It says that it is more (or at least
equally) beneficial to add a coalition to a larger coalition than to
a smaller one.

Assume that $i$ is not a member of some coalition $A$. The
\emph{marginal contribution} $m_i(A)$ of $i$ to the coalition $A$ is
the quantity
\[ m_i(A) = v(A \cup i ) - v(A), \]
where $A \cup i$ denotes the coalition $A  \cup \{i\}$. Therefore,
the marginal contribution of $i$ to $A$ measures the added value
obtained by bringing player $i$ into the coalition $A$.

A game is convex if and only if, for every player $i$, and all
coalitions $A \subseteq B$ that do not contain $i$,
\[ m_i(A) \leq m_i(B), \]
i.e., it is more beneficial to add a player to a larger coalition
than to a smaller one (this is a well known fact; see for
instance~\cite[Theorem~1.4.2]{curiel:b-games}
or~\cite[Theorem~4.9]{branzei-d-t:b-games}).

\begin{example}
The marginal contributions in the game from Example~\ref{e:4-player}
are written on the edges of the lattice of coalitions. For instance,
the fact that $m_2(\{1,3\}) = v(\{1,2,3\}) - v(\{1,3\}) = 4-2 =2$ is
indicated by the label 2 on the edge between $\{1,3\}$ and $\{1,3\}
\cup\{2\} = \{1,2,3\}$.
\end{example}

The top marginal contributions $m_1(N-\{1\}), \dots, m_n(N-\{n\})$
are often denoted by $m_1,\dots,m_n$. Further, we denote
\[
 M = \sum_{i \in N} m_i, \qquad
 T = v(N), \qquad
 V = \sum_{i \in N} v_i.
\]
Note that, in a convex game, $M \geq T \geq V$.

\begin{example}\label{e:3-player}
We provide a diagram for a general example of a game on three
players. The marginal contributions are indicated on the edges, Note
that, for $i,j \in N=\{1,2,3\}$, $m_i(\emptyset) = v_i$, and
whenever $i \neq j$, we denote $m_i\left(\{j\}\right) = m_{ij}$.
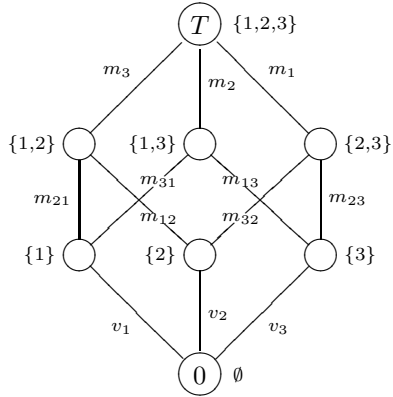
\begin{figure}[!ht]
\[
\xymatrix@C=30pt@R=30pt{
 &&*++[o][F-]{T} \ar@{-}[dl]_{m_3} \ar@{-}[d]^{m_2}
                 \ar@{-}[dr]^{m_1} \ar@{}[r]|<<<<<<{\{1,2,3\}} &&
 \\
 &*++[o][F-]{} \ar@{-}[d]_{m_{21}} \ar@{-}[dr]|>>>>>>{m_{12}} \ar@{}[l]|<<<<<{\{1,2\}}
 &*++[o][F-]{} \ar@{-}[dl]|<<<<<<{m_{31}} \ar@{-}[dr]|<<<<<<{m_{13}} \ar@{}[l]|<<<<<{\{1,3\}}
 &*++[o][F-]{} \ar@{-}[dl]|>>>>>>{m_{32}} \ar@{-}[d]^{m_{23}} \ar@{}[r]|<<<<<{\{2,3\}} &
 \\
 &*++[o][F-]{} \ar@{-}[dr]_{v_1} \ar@{}[l]|<<<<{\{1\}}
 &*++[o][F-]{} \ar@{-}[d]^{v_2}  \ar@{}[l]|<<<<{\{2\}}
 &*++[o][F-]{} \ar@{-}[dl]^{v_3} \ar@{}[r]|<<<<{\{3\}} &
 \\
 && *++[o][F-]{0} \ar@{}[r]|<<<{\emptyset}&&
}
\]
\caption{A game on 3 players}\label{f:3-players}
\end{figure}

Note that, if the triple $(i,j,k)$ is a permutation of $N$ then
\[ v_i + m_{ji} + m_k = T. \]
The convexity of the game is equivalent to the system of
inequalities
\begin{equation}\label{ineqs}
 v_i+v_j+m_k \leq T \leq m_i + m_j + v_k,
\end{equation}
where $(i,j,k)$ ranges over the permutations of $N$ (see the
appendix for details).
\end{example}

An \emph{efficient allocation} is a function $x:N \to \R$ such that
$x(N) = v(N)$. If in addition $x_i \geq v_i$, for $i \in N$, the
allocation is called \emph{individually rational}.

An efficient allocation assigns revenue to each player in the game
in such a way that the total revenue shared among the players is
exactly the value of the grand coalition $N$. The individual
rationality of an allocation then just means that each player should
be assigned revenue that is not below the individual value of that
player (otherwise that player would choose not to cooperate).

A convex game is \emph{essential} if $T > V$. In an inessential
game, $m_i(A)=v_i$, for all coalitions $A$ not containing the player
$i$, and there exists a unique efficient and individually rational
allocation, namely $x_i=v_i$, for all $i \in N$.

For a permutation $\pi$ of $N$, and a player $i$ in $N$, denote by
$P_i(\pi)$ the set of \emph{predecessors} of $i$ in $\pi$. This is
the set of players that appear before $i$ (to the left of $i$) in
the one-line representation of the permutation $\pi$. For instance,
if $n=6$ and $\pi = 142536$, then $P_3(\pi) = \{1,4,2,5\}$.

The set of permutations of $N$, denoted $\Pi_n$, represents all
possible orders in which the grand coalition can be formed by adding
the players one by one to the coalition. For each such order, the
players have different marginal contributions depending on the set
of players that has already joined. The \emph{marginal contribution}
of the player $i$ to the permutation $\pi$, denoted $m_i(\pi)$, is
the marginal contribution of the player $i$ to the coalition
$P_i(\pi)$ consisting of the predecessors of $i$ in $\pi$. In a
convex game, for any permutation $\pi \in \Pi_n$, we have $m_i(\pi)
\geq m_i(\emptyset) = v_i$ and $\sum_{i \in N} m_i(\pi) = T$. Thus
the marginal contribution vector along $\pi$ represents an efficient
and individually rational allocation for $v$.

An \emph{efficient solution function} $f$ is a function that assigns
an efficient allocation $f^v$ to every convex game (we emphasize
that we are not concerned with non-convex games).

Recall the definition of a well known efficient solution function
introduced by Shapley~\cite{shapley:value}.

\begin{definition}
The \emph{Shapley value} of a convex game $v:2^N \to \R$ is the
allocation $s$ given by
\[ s_i = \frac{1}{n!} \sum_{\pi \in \Pi_n} m_i(\pi).  \]
\end{definition}

Thus the Shapley value is the average of all marginal contribution
vectors along all permutations of $N$.

We also recall the definition of $\tau$-value, introduced by
Tijs~\cite{tijs:value}.

\begin{definition}
The $\tau$-\emph{value} of an essential convex game $v:2^N \to \R$
is the allocation given by
\[ \tau_i = \frac{M-T}{M-V}\ v_i + \frac{T-V}{M-V}\ m_i .\]

In the case of an inessential game, the $\tau$ value is the unique
efficient and individually rational allocation.
\end{definition}

Note that, for an essential game, $\tau_i = \lambda v_i +
(1-\lambda) m_i$, where $\lambda=\frac{M-T}{M-V}$ is the unique real
number in $[0,1]$ making the allocation efficient. For an
inessential game, $v_i = m_i$, for all $i$, and therefore the
formula $\tau_i = \lambda v_i + (1-\lambda) m_i$ gives the correct
$\tau$-value for all $\lambda$ in the interval $[0,1]$, i.e., the
normalizing coefficient $\lambda$ is not unique.

Both the Shapley value and the $\tau$-value are efficient solution
functions that assign an individually rational allocation to every
convex game.

\begin{example}
For any convex 2-player game,
\[ s_1 = \tau_1 = \frac{1}{2}(T+v_1-v_2), \qquad s_2 = \tau_2 = \frac{1}{2}(T+v_2-v_1). \]
\end{example}

\section{Encouraging the grand coalition}

We come to our main definition.

\begin{definition}
An efficient solution function $f$ \emph{encourages the grand
coalition} if for every convex game $v: 2^N \to \R$ and every
coalition $A \subseteq N$,
\[ f^v_i \geq f^{v_A}_i, \]
where $v_A:2^N \to \R$ is the convex sub-game of $v$ obtained by
restriction on the coalition $A$.
\end{definition}

Thus an efficient solution functions encourages the grand coalitions
if no player in any convex game would prefer any coalition (and its
associated allocation) over the grand coalition. If $f$ is an
efficient solution that encourages the grand coalition and if all
players were to vote for all coalitions they like (based on
maximizing the revenue they would obtain by applying the proposed
solution function $f$) the grand coalition would be chosen by each
player (even though some players may like some additional choices).

Note that the property of encouraging the grand coalition is a
global property of solution functions and not of individual
allocations (the property requires that we compare allocations in
different games).

\begin{theorem}
The Shapley value encourages the grand coalition in convex games.
\end{theorem}

\begin{proof}
Without loss of generality, it is sufficient to show that player 1
does not prefer any coalition $M=\{1,\dots,m\}$ to the grand
coalition, i.e., it is sufficient to show that
\[
 \frac{1}{n!} \sum_{\pi \in \Pi_n} m_1(\pi) \geq
 \frac{1}{m!} \sum_{\sigma \in \Pi_m} m_1(\sigma),
\]
for $1 \leq m \leq n$.

Define a map $~\bar{}:\Pi_n \to \Pi_m$ by flattening the
permutations of $N$ to permutations of $M$. Namely, for a
permutation $\pi \in \Pi_n$ define the permutation $\bar{\pi} \in
\Pi_m$ by deleting the symbols $m+1,\dots,n$ from $\pi$ and keeping
the relative order of the symbols $1,\dots,m$ the same as in $\pi$
(for instance, if $n=6$, $m=4$ and $\pi = 153462$, then
$\bar{\pi}=1342$). Every permutation in $\Pi_m$ is the image of
exactly $n!/m!$ permutations in $\Pi_n$ under the flattening map.

Note that the set of predecessors $P_1(\pi)$ of $1$ in the
permutation $\pi$ contains the set of predecessors $P_1(\bar{\pi})$
of $1$ in the flattened permutation $\bar{\pi}$. Therefore, by the
convexity of the game, $m_1(\pi) = m_1(P_1(\pi)) \geq
m_1(P_1(\bar{\pi}))= m_1(\bar{\pi})$.

It follows that
\[
 \frac{1}{n!} \sum_{\pi \in \Pi_n} m_1(\pi) \geq
 \frac{1}{n!} \sum_{\pi \in \Pi_n} m_1(\bar{\pi}) =
 \frac{1}{n!}\cdot \frac{n!}{m!} \sum_{\sigma \in \Pi_n} m_1(\sigma) =
 \frac{1}{m!} \sum_{\sigma \in \Pi_n} m_1(\sigma),
\]
which is what we needed to prove.
\end{proof}

\begin{theorem}\label{t:tau3}
The $\tau$-value encourages the grand coalition in all convex games
of up to three players.
\end{theorem}

\begin{proof}
Since the $\tau$-value coincides with the Shapley value for 2-player
convex games and Shapley value encourages the grand coalition, it
suffices to consider only 3-player games.

Further, the $\tau$-value always produces individually rational
allocations. Thus, it suffices to consider only 3-player games and
their 2-player sub-games.

By symmetry, it suffices to show that the convexity of a 3-player
game $v$ on $N=\{1,2,3\}$ implies the inequality
\[
 \tau^v_1 \geq \tau^{v'}_1,
\]
where $v'$ is the sub-game corresponding to the coalition
$A=\{1,2\}$.

If the game $v$ is inessential then so is its sub-game $v'$ and the
$\tau$-values for $v$ and $v'$ agree on $A$.

Thus we may assume that $v$ is essential and we need to show that
the convexity of $v$ implies
\begin{equation}\label{needtoshow}
 \frac{M-T}{M-V}\ v_1 + \frac{T-V}{M-V}\ m_1
     \geq \frac{1}{2}(T'+v_1-v_2),
\end{equation}
where $T' = v(A) = T-m_3$ is the value of the coalition $A=\{1,2\}$.

Denote $m_{12} = T'-v2 = T-m_3-v2$ (Figure~\ref{f:3-players} may be
useful for visualization; the marginal contribution vector along the
permutation $\pi=213$ is important in our considerations). Taking
into account that $M
> V$ (from the fact that $v$ is convex and essential) the
inequality~\eqref{needtoshow} takes the form
\[
 (M-T)v_1 + (T-V)m_1 \geq \frac{1}{2}(T'+v_1-v_2)(M-V),
\]
which is equivalent to
\[   v_1(m_1+m_3+v_2-v_1-v_3-m_2) + m_{12}(m_2+m_3+v_1-v_2-v_3-m_1) \leq 2m_1(m_3-v_3), \]
after substituting $V=v_1+v_2+v_3$, $M = m_1+m_2+m_3$, $T = v_2 +
m_{12}+ m_3$, and $T' = v_2 +m_{12}$, and performing simple
algebraic manipulations. The convexity implies that $v_1 \leq m_{12}
\leq m_1$, as well as that $m_1+m_3+v_2-v_1-v_3-m_2 \geq 0$ and
$m_2+m_3+v_1-v_2-v_3-m_1 \geq 0$ (see the inequalities
in~\eqref{ineqs}). Thus
\begin{multline*}
   v_1(m_1+m_3+v_2-v_1-v_3-m_2) + m_{12}(m_2+m_3+v_1-v_2-v_3-m_1) \leq \\
   \leq m_1(m_1+m_3+v_2-v_1-v_3-m_2+m_2+m_3+v_1-v_2-v_3-m_1) = 2m_1(m_3-v_3),
\end{multline*}
which is what we needed to prove.
\end{proof}

\begin{example}\label{e:tau-not-good}
Consider again the convex game in Example~\ref{e:4-player}. This
example shows that the $\tau$-value does not necessarily encourage
the grand coalition for convex 4-player games.

Indeed, we have $T=6$, $V=0$, $M=11$, which shows that the
normalizing coefficient $\lambda$ in the formula for the
$\tau$-value is $\lambda = (M-T)/(M-V)=5/11$. Direct calculation
then gives the $\tau$-values for $v$
\[
 \tau^v_1 = \frac{18}{11} \approx 1.64, \qquad \tau^v_2 = \frac{18}{11} \approx 1.64, \qquad
 \tau^v_3 = \frac{18}{11} \approx 1.64, \qquad \tau^v_4 = \frac{12}{11} \approx 1.09~.
\]
On the other hand, for the 3-player sub-game $v'$ determined by the
coalition $A=\{1,2,3\}$, we have $T'=4$, $V'=0$, $M'=7$,
$\lambda'=3/7$ and the $\tau$-values for $v'$ are
\[
 \tau^{v'}_1 = \frac{12}{7} \approx 1.71. \qquad
 \tau^{v'}_2 = \frac{8}{7} \approx 1.14, \qquad
 \tau^{v'}_3 = \frac{8}{7} \approx 1.14~.
\]
Thus player 1 would prefer the coalition $A$ to the grand coalition,
showing that the $\tau$-value does not necessarily encourage the
grand coalition.
\end{example}

\section{Relation to the core}

We consider the relation between efficient solution functions that
encourage the grand coalition and the core of a convex game.

\begin{definition}
The \emph{core} of a convex game $v: 2^N \to \R$ is the set of all
efficient allocations $x:N \to \R$ such that, for every coalition $A
\subseteq N$,
\[ x(A) \geq v(A). \]
\end{definition}

Note that every allocation in the core is individually rational and
we may say that the core allocations are rational with respect to
any coalition.

\begin{proposition}\label{p:core}
Let $f$ be an efficient solution function that encourages the grand
coalition in convex games. Then, for every convex game $v$, the
allocation $f^v$ is in the core of $v$.
\end{proposition}

\begin{proof}
Let $f$ be an efficient solution function that encourages the grand
coalition in convex games and let $v$ be a convex game. Then, for
any coalition $A$,
\[
 f^v(A) = \sum_{i \in A} f^v_i \geq \sum_{i \in A} f^{v_A}_i = v(A).
\]
Thus $f^v$ is in the core of $v$.
\end{proof}

Since the $\tau$-value does not always produce allocations in the
core of a convex function, we could immediately see that it cannot
encourage the grand coalition in general. However, in
Example~\ref{e:tau-not-good} the $\tau$-value of the game $v$ on $N$
is in the core, as is the $\tau$-value of all of its sub-games, but
this was still not sufficient to encourage the grand coalition.

The proof of Proposition~\ref{p:core} indicates that, for essential
solution functions, the property of encouraging the grand coalition
is a refinement of the property of producing solutions in the core.
Indeed, the core condition requires that, for each coalition $A
\subseteq N$, the sum $f^v(A) = \sum_{i \in A} f^v_i$ is at least as
large as the sum $\sum_{i \in A} f^{v_A}_i= v(A)$. On the other
hand, for solution functions that encourage the grand coalition,
each term in the former sum must be at least as large as the
corresponding term in the latter sum. In order to see that this
refinement is proper, we provide an example of an efficient solution
function that always produces allocations in the core of convex
games, but nevertheless fails to encourage the grand coalition.

\begin{example}
For any convex game $v$ and any permutation $\pi$ of $N$, the vector
of marginal contributions along $\pi$ is an efficient solution in
the core of $v$. By convexity of the core, any convex linear
combination of marginal contributions along several permutations is
also in the core. Therefore, we may define an efficient solution
function $f$ as follows. Among all permutations of $N$ select those
that give the largest vectors (in the usual sense in $\R^n$) of
marginal contributions and calculate their average. Thus, if
\[
 L = \left\{\ \pi \in \Pi_n \mid \sum_{i \in N} m_i^v(\pi)^2 \geq \sum_{i \in N} m_i^v(\sigma)^2,
                            \text{ for all } \sigma \in \Pi_n \ \right\},
\]
define
\[ f^v_i  = \frac{1}{|L|} \sum_{\pi \in L} m_i^v(\pi). \]

To see that $f$ does not encourage the grand coalition in convex
games, even though it always produces allocations in the core,
consider the game in Example~\ref{e:4-player} restricted to
$N=\{1,2,3\}$ (completely ignore player 4).

In this game, the largest marginal vectors are the two vectors along
$\pi_1 = 231$ and $\pi_2 = 321$ giving
\[
 f^v_1 = 3, \qquad f^v_2 = \frac{1}{2}, \qquad f^v_3 = \frac{1}{2}.
\]
On the other hand, if we restrict to the sub-game $v'$ defined by
the coalition $A=\{1,2\}$ we obtain
\[
 f^{v'}_1 = 1, \qquad f^{v'}_2 = 1.
\]
Thus player 2 would prefer the coalition $A$ to the grand coalition.
\end{example}

\section{Relation to population monotone allocation schemes}

The notion of a population monotone allocation scheme was introduced
in~\cite{sprumont:monotonic}.

Given a game $v$, a monotone allocation scheme is a set of efficient
allocations $\{x^{v_A} \mid A \subseteq N\}$ associated to the
sub-games of $v$, in such a way that, for every player $i$ and all
coalitions $A$ and $B$ with $i \in A \subseteq B \subseteq N$,
\[ x_i^{v_A} \leq x_i^{v_B}. \]

This definition is close in spirit to our definition of solution
functions that encourage the grand coalition. However, the emphasis
goes in different direction. We study efficient solution functions
that behave well on convex games, while Sprumont studies games for
which well behaved allocation schemes exist. More precisely, the
main thrust of Sprumont's work is a characterization of games for
which population monotone allocation schemes exist (this includes
all convex games, but not all games with non-empty core). For us, on
the other hand, the important question is which solution functions
always produce (or fail to produce) population monotone allocation
schemes in all convex games.

Sprumont shows that every 3-player game that is totally balanced
(see the appendix for a definition) always has a population monotone
allocation scheme. Nevertheless, Theorem~\ref{t:tau3} does not
follow directly from this observation (we still need to prove that
the specific scheme induced by the $\tau$-value solution function
provides such an allocation scheme).

Further, Sprumont shows that the glove game on 4 players fails to
have a population monotonic allocation scheme. Again, this example
is not helpful in our considerations, since the glove game is not
convex (the main point of Example~\ref{e:tau-not-good} is that the
$\tau$-value fails to provide a population monotone allocation
scheme on a convex game; on the other hand this game certainly has a
population monotone allocation scheme, namely the one induced by the
Shapley value).

\section{On necessity versus desirability}

Observe that even if a solution function that does not encourage the
grand coalition is used and, for a concrete game $v$, there exists a
player that prefers some smaller coalition over the grand one, this
does not mean that the grand coalition will not be formed. For
instance, in Example~\ref{e:tau-not-good} player 1 prefers
$A=\{1,2,3\}$ to $N$, but will have difficulties convincing player 2
and player 3 to form this coalition, since they certainly prefer the
payout provided to them by the grand coalition. Therefore player 1
would perhaps choose to join the grand coalition (however
grudgingly), since it is still offering a better payoff than
going-it-alone (which would bring a payoff of 0 to player 1).
However, even if player 1 joins the grand coalition, it would be
unsatisfied with the situation and may show its discontent by
actively and visibly (or covertly and by using inappropriate means)
working to undermine the grand coalition and exclude player 4.

Thus encouraging the grand coalition is not necessary to coalescence
all players into the grand coalition, but may be desirable in
practice.

\subsection*{Acknowledgments} The authors would like to thank
Imma Curiel, who provided valuable suggestions in the early stages
of the manuscript preparation, and Iurie Boreico, who did the same
in the final stages.

\appendix

\section{Remarks on 3-player games}

In Example~\ref{e:3-player} we provided a quick remark on a
condition on 3-player games that is equivalent to convexity and we
used this condition in the course of the proof of
Theorem~\ref{t:tau3}.  We provide a brief justification.

\begin{proposition}
A 3-player game $v$ is convex if and only if, for every permutation
$(i,j,k)$ of $N$,
\begin{equation}\label{ineqs2}
 v_i+v_j+m_k \leq T \leq m_i + m_j + v_k.
\end{equation}
\end{proposition}

\begin{proof}
As we already remarked, a game is convex if and only if, for every
player $i$, and all coalitions $A \subseteq B$ that do not contain
$i$, $m_i(A) \leq m_i(B)$.

Therefore, in the context of a 3-player game the convexity is
equivalent to the system of inequalities
\begin{equation}\label{e:system}
 v_i \leq m_{ij} \leq m_i,
\end{equation}
for $i,j \in N$, $i \neq j$.

The inequality~\eqref{e:system} is equivalent to
\[ v_i+v_j+m_k \leq v_j+ m_{ij} + m_k \leq m_i +v_j + m_k, \]
where $k$ is the third player (different from $i$ and $j$). Since
$T=v_j+m_{ij}+m_k$ we obtain
\[ v_i+v_j+m_k \leq T \leq m_i +v_j + m_k. \]

Thus, when looked as systems of inequalities, \eqref{ineqs2} and
\eqref{e:system} are equivalent.
\end{proof}

The games with non-empty core were characterized by
Bondareva~\cite{bondareva:core}. Namely, a game has a non-empty core
if and only if it is balanced. A game $v$ is balanced if, for every
sequence of non-empty subsets $A_1,\dots,A_s$ of $N$ and every
sequence of positive real numbers $\lambda_1,\dots,\lambda_s$ such
that
\begin{equation}\label{e:=chi}
 \sum_{\ell=1}^s \lambda_\ell \chi_{A_\ell} = \chi_N,
\end{equation}
where $\chi_{A_\ell}$ and $\chi_N$ denote the characteristic
function of the sets $A_\ell$ and $N$, we have
\[ \sum_{\ell=1}^s \lambda_\ell v(A_\ell) \leq v(N). \]
A game is totally balanced if all of its sub-games are balanced.

The following modification of the balancing condition is also valid.

\begin{proposition}\label{p:bondareva}
A game $v$ has non-empty core if and only if, for every sequence of
non-empty subsets $A_1,\dots,A_s$ of $N$ and every sequence of
positive real numbers $\lambda_1,\dots,\lambda_s$ such that
\begin{equation}\label{e:leq-chi}
 \sum_{\ell=1}^s \lambda_\ell \chi_{A_\ell} \leq \chi_N,
\end{equation}
where the inequality is considered pointwise, we have
\begin{equation}\label{e:leqN}
 \sum_{\ell=1}^s \lambda_\ell v(A_\ell) \leq v(N).
\end{equation}
\end{proposition}

\begin{proof}
Let $x$ be an efficient allocation in the core of $v$, and let
$\sum_{\ell=1}^s \lambda_\ell \chi_{A_\ell} \leq \chi_N$, for some
positive real numbers $\lambda_1,\dots,\lambda_s$ and a sequence of
non-empty subsets $A_1,\dots,A_s$ of $N$. We have
\begin{align*}
 \sum_{\ell=1}^s \lambda_\ell v(A_\ell) &\leq
 \sum_{\ell=1}^s \lambda_\ell x(A_\ell) =
 \sum_{\ell=1}^s \lambda_\ell \sum_{i \in A_\ell} x_i =
 \sum_{\ell=1}^s \lambda_\ell \sum_{i=1}^n \chi_{A_\ell}(i)x_i = \\
 &= \sum_{i=1}^n \left( \sum_{\ell=1}^s \lambda_\ell\chi_{A_\ell}(i)\right)x_i \leq
 \sum_{i=1}^n x_i = v(N).
\end{align*}
The other direction follows from the result of Bondareva. Namely,
if~\eqref{e:leqN} holds whenever~\eqref{e:leq-chi} does,
then~~\eqref{e:leqN} also holds whenever~\eqref{e:=chi} does.
Therefore the core of $v$ is non-empty.
\end{proof}

It is easy to see that for a 2-player game, convexity,
super-additivity,  and the existence of core allocations are
equivalent properties and it is well known that these properties are
not equivalent for more than 2 players.

\begin{proposition}
Let $v$ be a 3-player super-additive game. Define $M_{12}=v_{12} -
v_1-v_2$, $M_{13}=v_{13} - v_1-v_3$, $M_{23}=v_{23} - v_2-v_3$, and
$S= v(N) - v_1-v_2-v_3$, where $v_{ij}$ is the value of the
coalition $\{i,j\}$.

(a) The game $v$ has a non-empty core if and only if
\begin{equation}\label{1/2}
 S \geq \frac{1}{2}(M_{12} + M_{13} + M_{23}).
\end{equation}

(b) The game $v$ is convex if and only if
\begin{equation}\label{pairs}
 S \geq \max \{ M_{12} + M_{13},\ M_{12} + M_{23},\ M_{13}+M_{23} \ \}.
\end{equation}
\end{proposition}

\begin{proof}
(a) Assume $v$ has a non-empty core. By
Proposition~\ref{p:bondareva} (or directly by the argument used in
the proof), since $\chi_{12} + \chi_{23} + \chi_{23} = 2\chi_N$, we
obtain that $v_{12}+v_{13}+v_{23} \leq 2v(N)$. Therefore, $M_{12} +
M_{13} + M_{23} = v_{12}+v_{13}+v_{23} - 2(v_1+v_2+v_3) \leq 2v(N) -
2(v_1+v_2+v_3) = 2S$.

Conversely, assume that~\eqref{1/2} holds. Instead of trying to use
Proposition~\ref{p:bondareva}, we construct explicitly an element in
the core.

Assume that the sum of every pair of numbers from
$\{M_{12},M_{13},M_{23}\}$ is no smaller than the third one
(triangle-like inequalities hold). Set
$a_1=\frac{M_{12}+M_{13}-M_{23}}{2}$,
$a_2=\frac{M_{12}+M_{23}-M_{13}}{2}$,
$a_3=\frac{M_{13}+M_{23}-M_{12}}{2}$, and
$t=\frac{1}{3}\left(S-\frac{1}{2}\left(M_{12}+M_{13}+M_{23}\right)\right)$.
Then $a_1$, $a_2$, $a_3$, and $t$ are non-negative. Set $x_1 =
v_1+a_1+t$, $x_2 = v_2+a_2+t$, and $x_3 = v_3+a_3+t$. Since
$x_1+x_2+x_3 = v(N)$, the allocation $x$ is efficient. The
allocation $x$ is individually rational (by the non-negativity of
$a_1$, $a_2$, $a_3$, and $t$). We also have
\[ x_1+x_2 = v_1 + v_2 + M_{12} + 2t = v_{12}+2t \geq v_{12}. \]
Thus the allocation $x$ is rational for the coalition $\{1,2\}$. By
symmetry, $x$ is rational for the other two 2-element coalitions as
well. Note that we have not used yet the super-additivity property.

Assume that the sum of two of the numbers $M_{12},M_{13},M_{23}$ is
smaller than the third, say $M_{12}>M_{13}+M_{23}$ and set
$t=\frac{1}{2}(S-M_{13}-M_{23})$. The super-additivity implies that
$v(N) \geq v_{12}+v_3$. Therefore $S = v(N) - v_1 -v_2 -v_3 \geq
v_{12}+v_3 -v_1-v_2-v_3 = M_{12}$. Since $S \geq M_{12}>
M_{13}+M_{23}$, we have that $t>0$. Set $x_1 = v_1+M_{13}+t$, $x_2 =
v_2+M_{23}+t$, and $x_3 = v_3$. Since $x_1+x_2+x_3 = v(N)$, the
allocation $x$ is efficient. The allocation $x$ is individually
rational by the non-negativity of $M_{12}$, $M_{13}$, $M_{23}$, and
$t$ (for $i \neq j$, $M_{ij}$ is non-negative by the
super-additivity property). Further,
\[ x_1+x_3 = v_1+v_3+M_{13}+t = v_{13}+t \geq v_{13} \]
and, by symmetry,
\[ x_2+x_3 \geq v_{23}. \]
We also have
\[ x_1+x_2 = v_1+M_{13}+t + v_2+M_{23}+t = v_1+v_2 + S = v_{12}-M_{12} +S \geq v_{12}.\]
Thus the allocation $x$ is rational for all 2-element coalitions.

(b) Note that the convexity needs to be checked only for coalitions
that are not comparable (the convexity condition is trivially
satisfied when one of the coalitions is included in the other).
Therefore, given the super-additivity of the game, $v$ is convex if
and only if, for every permutation $(i,j,k)$ of $N$
\[ v_{ij} + v_{ik} \leq v(N) + v_i. \]
The last inequality is equivalent to
\[ M_{ij}+M_{ik} \leq S. \qedhere \]
\end{proof}

Therefore, we see that the convexity and the existence of the core
are not equivalent for 3-player games even in the presence of
super-additivity. For instance, if $v_1=v_2=v_3=0$,
$v_{12}=v_{13}=v_{23}=1$ and $v_{N}=3/2$, we have a super-additive,
non-convex game with non-empty core.

%---------------------------------------------------------------

\def\cprime{$'$}

%---------------------------------------------------------------

%\bibliographystyle{alpha}
%\bibliography{../smath}

\begin{thebibliography}{BDT05}

\bibitem[BDT05]{branzei-d-t:b-games}
Rodica Branzei, Dinko Dimitrov, and Stef Tijs.
\newblock {\em Models in cooperative game theory}, volume 556 of {\em Lecture
  Notes in Economics and Mathematical Systems}.
\newblock Springer-Verlag, Berlin, 2005.
\newblock Crisp, fuzzy, and multi-choice games.

\bibitem[Bon63]{bondareva:core}
O.~N. Bondareva.
\newblock Some applications of the methods of linear programming to the theory
  of cooperative games.
\newblock {\em Problemy Kibernet. No.}, 10:119--139, 1963.

\bibitem[Cur97]{curiel:b-games}
Imma Curiel.
\newblock {\em Cooperative game theory and applications}, volume~16 of {\em
  Theory and Decision Library. Series C: Game Theory, Mathematical Programming
  and Operations Research}.
\newblock Kluwer Academic Publishers, Boston, MA, 1997.
\newblock Cooperative games arising from combinatorial optimization problems.

\bibitem[Sha53]{shapley:value}
L.~S. Shapley.
\newblock A value for {$n$}-person games.
\newblock In {\em Contributions to the theory of games, vol. 2}, Annals of
  Mathematics Studies, no. 28, pages 307--317. Princeton University Press,
  Princeton, N. J., 1953.

\bibitem[Spr90]{sprumont:monotonic}
Yves Sprumont.
\newblock Population monotonic allocation schemes for cooperative games with
  transferable utility.
\newblock {\em Games Econom. Behav.}, 2(4):378--394, 1990.

\bibitem[Tij81]{tijs:value}
Stef~H. Tijs.
\newblock Bounds for the core and the $\tau$-value.
\newblock In O.~Moeschlin and D.~Pallaschke, editors, {\em Game Theory and
  Mathematical Economics}, pages 123--132. North Holland, Amsterdam, 1981.

\end{thebibliography}

\end{document}